\documentclass[a4paper,12pt]{amsart}

\usepackage{amsthm}
\usepackage[pagebackref]{hyperref} 
\hypersetup{ colorlinks=true, linkcolor=magenta, urlcolor=red, citecolor=blue }

\usepackage[normalem]{ulem} 
\usepackage{graphicx,xcolor}
\usepackage{dsfont}
\usepackage{enumitem}
\usepackage{tikz}
\usepackage{tikz-cd}
\usepackage[utf8]{inputenc}
\usepackage[T1]{fontenc}
\usepackage[all,cmtip]{xy}
\usepackage{todonotes}
\usepackage[titletoc]{appendix}
\usepackage{color}
\usepackage{amsmath,amssymb,amscd,amsfonts}
\usepackage{mathtools}
\usepackage{babel}
\usepackage{amstext}
\usepackage{amsmath}
\usepackage{amsfonts}
\usepackage{latexsym}
\usepackage{ifthen}
\usepackage{yfonts}
\usepackage[all,cmtip]{xy}
\xyoption{all}
\usepackage{enumitem}

\usepackage[pagebackref]{hyperref}
\hypersetup{
  colorlinks=true,
  linkcolor=magenta,
  urlcolor=red,
  citecolor=blue
}

\setlist[enumerate]{label=(\thethm.\arabic*), before={\setcounter{enumi}{\value{equation}}}, after={\setcounter{equation}{\value{enumi}}}}

\newtheorem{thm}{Theorem}[section]
\newtheorem{theo}[thm]{Theorem}
\newtheorem{cor}[thm]{Corollary}
\newtheorem{prop}[thm]{Proposition}
\newtheorem{conj}[thm]{Conjecture}

\newtheorem{lem}[thm]{Lemma}
\theoremstyle{definition}
\newtheorem{defn}[thm]{Definition}

\newtheorem{ques}[thm]{Question}
\theoremstyle{remark}
\newtheorem{rem}[thm]{Remark}
\newtheorem{prob}[thm]{Problem}
\newtheorem{setup}[thm]{Setup}

\newtheorem{claim}[thm]{Claim}

\numberwithin{equation}{section}

\newcommand{\GL}[0]{\operatorname{GL}}

\newcommand{\codim}[0]{\operatorname{codim}}

\newcommand{\rank}[0]{\operatorname{rank}}

\newcommand{\id}{{\rm id}}
\newcommand{\Sym}{{\rm Sym}}

\newcommand{\num}{{\rm{num}}}
\newcommand{\tf}{\rm{tf}}
\newcommand{\orb}{\rm{orb}}

\newcommand{\reg}{{\rm{reg}}}
\newcommand{\sing}{{\rm{sing}}}

\newcommand{\ddbar}{dd^c}

\newcommand{\Ker}[1]{\mathrm{Ker}(#1)}
\newcommand{\Image}[1]{\mathrm{Im}(#1)}

\title[compact K\"ahler varieties with psef tangent sheaf]
{A  note on compact K\"ahler varieties \\ 
with pseudo-effective tangent sheaf}

\author{Shin-ichi MATSUMURA}
\address{Mathematical Institute
$\&$ Division for the Establishment of Frontier Science of Organization for Advanced Studies,
Tohoku University,
6-3, Aramaki Aza-Aoba, Aoba-ku, Sendai 980-8578, Japan.}
\email{{\tt mshinichi-math@tohoku.ac.jp}}
\email{{\tt mshinichi0@gmail.com}}

\author{Guolei ZHONG}
\address{
\textsc{School of Mathematical Sciences, Shanghai Key Laboratory of PMMP}\endgraf 
\textsc{East China Normal University, Shanghai 200241, China}
}
\email{glzhong@math.ecnu.edu.cn}

\thanks{S.\,M.\,was supported by the JST FOREST Program ($\sharp$PMJFR2368).
G.\,Z.\,was supported by the Science and Technology Commission of Shanghai Municipality (No. 22DZ2229014), and the National Natural Science Foundation of China. }



\begin{document}

\begin{abstract}
In this paper, we discuss recent developments and open problems 
concerning the geometry of compact K\"ahler varieties whose tangent sheaves are pseudo-effective in a strong sense.
As our main result, we prove that, after passing to a finite quasi-\'etale cover, any compact K\"ahler variety with quotient singularities and pseudo-effective tangent sheaf admits a flat fibration onto a complex torus, with rationally connected fibers.
We also obtain an analogous structure theorem for klt compact K\"ahler varieties, assuming the existence of minimal models in numerical dimension zero; 
consequently, the result holds unconditionally in the projective setting and for low-dimensional compact K\"ahler varieties.
\end{abstract}

\maketitle

\section{Introduction} 
Positivity properties of tangent sheaves are known to impose strong restrictions on the global geometry of  projective varieties (or more generally, compact K\"ahler varieties). 
At the strongest end, the ampleness of the tangent bundle characterizes projective spaces \cite{SY80, Mor79}. 
Under weaker positivity assumptions on the tangent bundle, the foundational works \cite{HSW81, Mok88, CP91,DPS94} show that if a compact K\"ahler manifold $X$ has nef tangent bundle, or semi-positive holomorphic bisectional curvature, then, after passing to a finite \'etale cover, $X$ admits a locally constant fibration $X \to Y$ onto a complex torus whose fibers are Fano manifolds. 
We refer the reader to \cite{Cao19,CH19,CCM21,MW25} for results on nef anti-canonical divisors, and to  \cite{HW20, Mat22a, Mat26, Yan18} for results on semi-positive holomorphic sectional curvature.

In this paper, we focus on the weaker positivity condition that the tangent
sheaf is \textit{pseudo-effective} in the sense of Definition \ref{def-psef}. 
We shall study the structure of compact K\"ahler varieties with pseudo-effective tangent sheaf,  and discuss several open problems on how this positivity condition affects their global structure. 
For line bundles, the pseudo-effectivity admits an analytic formulation
in terms of singular Hermitian metrics with semipositive curvature current.
For vector bundles (or more generally, for torsion-free sheaves), the analogous
notion is more subtle; we adopt the following definition.

\begin{defn}[{\cite[Definition 2.1]{Mat23}}]\label{def-psef}
Let \(X\) be a normal compact K\"ahler variety.
Fix a K\"ahler form \(\omega_X\) on \(X\), 
that is, a positive $(1,1)$-form with local potential. 
A torsion-free sheaf \(\mathcal{E}\) on \(X\) is \textit{pseudo-effective} if, for every \(m\in\mathbb{Z}_+\), there exists a singular Hermitian metric \(h_m\) on the \(m\)-th symmetric power \(\textup{Sym}^m\mathcal{E}|_{X_0}\) such that \(\sqrt{-1}\Theta_{h_m}\geq-\omega_X\otimes\id_{\textup{Sym}^m\mathcal{E}}\). 
Here \(X_0\coloneqq X_{\textup{reg}}\cap X_{\mathcal{E}}\), where \(X_{\textup{reg}}\) is the non-singular locus of \(X\) and \(X_{\mathcal{E}}\) is the maximal subset where \(\mathcal{E}\) is locally free.
\end{defn}
Note that when \(X\) is normal and \(\mathcal{E}\) is torsion free, the Zariski open subset \(X_0\subseteq X\) above has the complement of codimension \(\geq 2\).

The pseudo-effectivity of the tautological line bundle
$\mathcal{O}_{\mathbb{P}(\mathcal{E})}(1)$ of the Grothendieck projectivization \(\mathbb{P}(\mathcal{E})\) 
is much weaker than the pseudo-effectivity of $\mathcal{E}$ in the sense of Definition \ref{def-psef}.
In the projective case, we have the following characterization of pseudo-effectivity  in terms of ample divisors.

\begin{prop}[{\cite[Proposition 2.4]{Mat23}}]
Let \(\mathcal{E}\) be a torsion-free sheaf on a normal projective variety \(X\).
Then, conditions \textup{(a)} and \textup{(b)} below are equivalent.
Moreover, if \(\mathcal{E}\) is locally free, then they are also equivalent to condition \textup{(c)}.
\begin{itemize}
\item[$($a$)$] \(\mathcal{E}\) is pseudo-effective in the sense of Definition \ref{def-psef}.
\item[$($b$)$] There exists an ample line bundle \(A\) such that, for every \(m\in\mathbb{Z}_{+}\),
the reflexive hull \(\textup{Sym}^{[m]}\mathcal{E}\otimes A\) is generically globally generated.
\item[$($c$)$] The non-nef locus of the tautological line bundle
\(\mathcal{O}_{\mathbb{P}(\mathcal{E})}(1)\) does not dominate \(X\)
under the natural projection \(\mathbb{P}(\mathcal{E}) \to X\).
\end{itemize}
\end{prop}

We refer the reader to \cite{BKK, IMZ23, Mat23} and the references therein for the basic properties of pseudo-effective sheaves.

For a compact K\"ahler \textit{manifold} $X$ with pseudo-effective tangent bundle,
the works \cite{HIM22,MQ26} show that, after passing to a finite \'etale cover, 
the manifold $X$ admits a smooth maximal rationally connected $($MRC$)$ fibration
$X \to Y$ onto a complex torus $Y$.
As posed in \cite{Mat23}, these results naturally lead to the following problems.

\begin{prob}\label{prob-smooth}
Let $X$ be a compact K\"ahler manifold with pseudo-effective tangent bundle.

\begin{enumerate}
\item[(1)] Let $X \to Y$ be the smooth MRC fibration of $X$ described above.
Is $X \to Y$ locally trivial?

\item[(2)] If moreover $X$ is rationally connected, 
is the anti-canonical divisor $-K_{X}$ big?
\end{enumerate}
\end{prob}

Let us recall why these problems are natural.
For \textup{(1)}, if the tangent bundle $T_X$ is singular positively curved,
that is, if it admits a singular semi-positively curved Hermitian metric, 
then the MRC fibration is locally constant (see \cite{Mul25,MQ26,HIM22}), which is stronger than local triviality.
On the other hand, local constancy cannot be expected under the mere pseudo-effectivity, 
as shown by the examples in \cite{HIM22}.
Nevertheless, local triviality remains open.

For \textup{(2)}, the work \cite{DPS94} shows that
a rationally connected manifold with nef tangent bundle is Fano,
and it is expected to be rational homogeneous by the Campana--Peternell conjecture.
It is thus natural to ask whether pseudo-effectivity of the tangent bundle
still forces some weaker form of Fano or homogeneous behavior.
A rationally connected manifold $X$ with pseudo-effective tangent bundle need not be Fano.
Moreover, \cite[Examples 6.4 and 6.5]{MZ} show that such a manifold $X$
is not necessarily of Fano type or almost homogeneous.
Nevertheless, it is still reasonable to expect that the anticanonical divisor $-K_X$ is big.

\smallskip
The main purpose of this paper is to study questions concerning singular compact K\"ahler \textit{varieties}.
In this generality, even the existence of an everywhere defined MRC fibration is not immediate,
and the structure of such a fibration, if exists, is much less understood than in the smooth case.
This leads us to the following problem.

\begin{prob}\label{prob-klt}
Let $X$ be a projective klt variety (or more generally, a compact klt K\"ahler variety) with pseudo-effective tangent sheaf.
\begin{enumerate}[label=\textup{(\arabic*)}]
\item Does there exist an everywhere defined MRC fibration $X \to Y$?

\item If such a fibration exists, what can one say about the base $Y$, the general fibers,
and the variation of the fibers?
\end{enumerate}
\end{prob}

Our first main result gives an affirmative answer 
for compact K\"ahler varieties with quotient singularities.

\begin{theo}\label{thm-quot}
Let $X$ be a compact K\"ahler variety with quotient singularities whose tangent sheaf is pseudo-effective.
Then, after replacing $X$ with a finite quasi-\'etale cover, there exists a fibration
$
\alpha \colon X \to Y
$
satisfying the following properties:
\begin{enumerate}
\item[$(1)$] The fibration $\alpha$ is flat, and
every irreducible component of the singular locus of $X$ dominates $Y$, if the singular locus is non-empty.
\item[$(2)$] The base $Y$ is an \'etale quotient of a complex torus, that is, there exists a finite \'etale cover $A \to Y$ by a complex torus $A$. 
\item[$(3)$] Every fiber of $\alpha$ is an irreducible, reduced, rationally connected klt projective variety, and a very general fiber has pseudo-effective tangent sheaf. 
\end{enumerate}
\end{theo}

The proof of Theorem \ref{thm-quot}  gives a slightly more precise geometric picture: 
There exists a resolution of singularities $\widehat{X} \to X$ such that the composite morphism
$
\widehat{X} \to X \to Y
$
is smooth.

In the case of (more general) klt singularities, the construction of such a well-behaved representative of the MRC fibration becomes substantially more delicate.
Nevertheless, we obtain the following structure theorem, conditional only on the expected existence of minimal models in numerical dimension zero.

\begin{theo}\label{thm-str}
Let $X$ be a compact klt K\"ahler variety with pseudo-effective tangent sheaf.
Assume that Conjecture \ref{conj-min} below holds.
Then, after replacing $X$ with a finite quasi-\'etale cover, there exists a fibration
$
X \to T
$
onto a complex torus $T$ such that a very general fiber $F$ is a rationally connected klt projective variety with pseudo-effective tangent sheaf.
\end{theo}

After the submission of our paper for a possible publication in the proceeding of AGEA2025, we obtain the full solution to the following conjecture so that Theorem \ref{thm-str} holds unconditionally. This will be included in our  forthcoming paper \cite{MZ}.

\begin{conj}\label{conj-min}
Let $Y$ be a compact K\"ahler manifold such that the numerical dimension of the canonical divisor $K_Y$ is zero. 
Then, there exists a bimeromorphic map
$
Y \dashrightarrow Y_{\min}
$
onto a compact klt K\"ahler variety $Y_{\min}$ such that $K_{Y_{\min}}$ is numerically trivial.
\end{conj}

For the numerical dimension of a (not necessarily nef) \((1,1)\)-class in the analytic case, we refer to \cite[(18.13) Definition]{Dem12}. 

Since Conjecture \ref{conj-min} is known in the projective case (see, e.g., \cite[Theorem 1.2]{Gon11} and \cite{Gon13}), 
and in low-dimensional K\"ahler cases,  
see \cite{CHP16, DHP24,DH25} as well as \cite[Corollary 5.9 and Theorem 5.11]{Wan21},
the following corollary follows immediately. 

\begin{cor}\label{cor-str}
Let $X$ be a klt compact K\"ahler variety with pseudo-effective tangent sheaf.
Assume further that $X$ is projective or that $\dim X \leq 4$.
Then, the conclusion of Theorem \ref{thm-str} holds.
\end{cor}

Finally, we relate our structure theorems to varieties \(X\) admitting an int-amplified endomorphism \(f\), i.e., \(f^*[\omega]-[\omega]\) contains a K\"ahler form for some K\"ahler class \([\omega]\). 
Let $X$ be a ($\mathbb{Q}$-factorial) projective klt variety with pseudo-effective tangent sheaf.
By running the minimal model program, we obtain a sequence of divisorial contractions, flips, and Fano contractions
$$
X\coloneqq X_{0} \dashrightarrow X_{1} \dashrightarrow X_{2} \dashrightarrow \cdots \dashrightarrow X_{N}
$$
such that each $X_i$ has klt singularities and pseudo-effective tangent sheaf. 
The final model $X_N$ is a quasi-\'etale quotient of an abelian variety (see \cite{Mat23}).

On the other hand, there is a similar phenomenon for projective varieties or compact K\"ahler varieties admitting an int-amplified endomorphism.
The existence of an int-amplified endomorphism gives strong restrictions on Albanese maps and MRC fibrations, as well as the equivariant minimal model program (see, for example, \cite{Meng20,MZ20,MY21,Yos21,Zho21}).
Thus, such endomorphisms can be viewed as a dynamical analogue of positivity. 
In a recent joint work of the authors \cite{MZ}, 
we show that varieties admitting an int-amplified endomorphism have pseudo-effective tangent sheaves.
The converse, however, is false in general (cf.~\cite[Corollary 1.4]{Yos21} and \cite[Example 6.5]{MZ}). 
This motivates us to pose the following question. 

\begin{prob}\label{prob-int}
Which parts of the structure theory for varieties admitting an int-amplified endomorphism remain valid
under the weaker assumption that the tangent sheaf is pseudo-effective?
\end{prob}

As a consequence of Theorems \ref{thm-quot}, \ref{thm-str} and \cite[Theorems 1.1 and 1.4]{MZ}, we obtain a structure theorem for a normal compact K\"ahler variety admitting an int-amplified endomorphism, provided either (1) \(X\) has only quotient singularities, or (2) \(X\) is klt and Conjecture \ref{conj-min} holds.
Moreover, if such a holomorphic MRC fibration exists, it is flat and every fiber is reduced and irreducible (see \cite[Lemmas 2.6 and 5.2]{Meng20}).

For the general picture, we conclude the introduction with the following question, an affirmative answer to which will establish the desired structure for klt compact K\"ahler varieties admitting an int-amplified endomorphism (see Corollary \ref{cor-int}).
\begin{ques}\label{q_mrc}
Let \(f\colon X\to X\) be a surjective endomorphism of a klt compact K\"ahler variety.
Does there exist a suitable MRC fibration \(\pi\colon X\dashrightarrow Y\) which is \(f\)-equivariant, that is, there exists a surjective endomorphism \(g\) on \(Y\) such that \(\pi\circ f^s=g\circ\pi\) for some integer \(s\)?
\end{ques}
The question has a positive answer when \(X\) is projective; see \cite[Theorem 4.19]{Nak10}.
However, it is still widely open in the K\"ahler setting.

\subsection*{Acknowledgements}
The first author is grateful to the organizers of the conference ``2025 Algebraic Geometry in East Asia,'' where a question raised during the discussion motivated him to revisit the study of pseudo-effective tangent sheaves.
He is also grateful to Professors Shou Yoshikawa and Kiwamu Watanabe
for discussions related to Problem \ref{prob-smooth} \textup{(2)}.
The second author would like to thank IBS-CCG for providing an excellent working environment and generous support over the past four years (2022-2026).
The authors thank the referee for the suggestions to improve the paper.

\section{Preliminary results}
In this section, we collect some preliminary results.
Although these results may be known to experts, we include the proofs for the reader's convenience.
For the basic notation and conventions involved, we refer the reader to \cite{HIM22,IMZ23,MQ26,CDM} and the references therein.

\begin{lem}\label{lem-Weil}
Let $X$ be a normal analytic variety $($that is, a normal, irreducible, and reduced analytic variety$)$
and let $\mathcal{L}$ be a reflexive sheaf of rank $1$ on $X$. 
Then, locally on $X$, there exists an effective Weil divisor $D$
such that $\mathcal{L} \cong \mathcal{O}_{X}(D)$.
\end{lem}

\begin{proof}
The assertion is local on $X$, and hence we may assume that $X$ is Stein.
By Cartan's theorem A, the sheaf $\mathcal{L}$ is generated by its global sections.
Hence, after shrinking $X$ if necessary, we can take a non-zero section
$
s \in H^0(X,\mathcal{L})
$
which generates $\mathcal{L}$ at a given point.
This determines the Weil divisor $D$ and 
the associated divisorial sheaf $\mathcal{O}_{X}(D)$ coincides with $\mathcal{L}$
since $\mathcal{L}$ is a reflexive sheaf of rank $1$.
\end{proof}

\begin{lem}\label{lem:q-f}
Let $X$ be a normal analytic variety.
Then \(X\) has quotient singularities if and only if, locally, there is a finite quasi-\'etale morphism \(Y\to X\) from a smooth variety \(Y\).
\end{lem}
\begin{proof}
Suppose first that \(X\) has  quotient singularities. 
Then we may assume that $X$ is sufficiently small 
 so that there exists a finite cover
$\tau \colon Y \to X\cong Y/G$ from a smooth variety $Y$ where \(G\) is a finite group acting on \(Y\).
By the classical Chevalley-Shephard-Todd theorem, with \(Y\) replaced by an \'etale quotient, we may further assume that \(G\) contains no pseudo-reflections so that \(\tau\) is quasi-\'etale. 

Suppose second that there is a finite quasi-\'etale morphism \(Y\to X\) from a smooth variety \(Y\).
We take the Galois closure \(\widetilde{Y}\to Y\to X\) which is still quasi-\'etale (see \cite[Theorem 3.7]{GKP16}).
By purity of the branch locus, \(\widetilde{Y}\to Y\) is \'etale and hence \(\widetilde{Y}\) is smooth; in particular, \(X\) has quotient singularities.
\end{proof}

\begin{lem}\label{lem-quot}
Let $X$ be a normal analytic variety with quotient singularities.
Let $X' \to X$ be a finite quasi-\'etale cover of $X$.
Then $X'$ also has quotient singularities.
\end{lem}
\begin{proof}
Let $\nu \colon X' \to X$ be the given finite quasi-\'etale cover.
Since \(X\) has only quotient singularities, by Lemma \ref{lem:q-f}, we may assume that $X$ is sufficiently small 
 so that there exists a finite quasi-\'etale cover
$\tau \colon Y \to X$ from a smooth variety $Y$.
Let $Z$ be the normalization of a main component of the fiber product $Y \times_X X'$.
We have the following commutative diagram:
\[
\xymatrix{
Z \ar[r]^{\mu} \ar[d]_{\tau'} & Y \ar[d]^{\tau} \\
X' \ar[r]_{\nu} & X .
}
\]
By construction, the induced morphism $\mu \colon Z \to Y$ is a finite quasi-\'etale cover.
Since $Y$ is smooth, the purity of the branch locus implies that $\mu$ is \'etale.
This shows that $Z$ is smooth as well.
Since \(\tau'\) is also quasi-\'etale, again by Lemma \ref{lem:q-f}, the variety \(X'\) has quotient singularities.
\end{proof}

\begin{lem}\label{lem-base}
Let $\phi \colon X \to Y$ be an equidimensional projective fibration between normal analytic varieties.
Then, the following statements hold:
\begin{itemize}
\item[$(1)$] If $X$ is strongly 
$\mathbb{Q}$-factorial, then so is $Y$.
\item[$(2)$] If $X$ has klt singularities and $K_Y$ is $\mathbb{Q}$-Cartier,
then $Y$ has klt singularities.
\end{itemize}
\end{lem}
\begin{proof}
(1)
Recall that $X$  is said to be \textit{strongly $\mathbb{Q}$}-factorial 
if for any reflexive sheaf $\mathcal{L}$ of  rank one, there is $m \in \mathbb{Z}_{+}$ 
such that $\mathcal{L}^{[\otimes m]}$ is an invertible sheaf (see \cite[Definition 2.2]{DH25}).
Let $\mathcal{L}$ be a reflexive sheaf of rank $1$ on $Y$.
The assertion is local on $Y$, 
so we may assume that $\mathcal{L}$ is a divisorial sheaf $\mathcal{O}_{Y}(D)$ by Lemma \ref{lem-Weil}.
Moreover, we may assume that $Y$ is a sufficiently small open set
so that $\phi \colon X \to Y$ is relatively embedded into the projection $\mathbb{P}^{N} \times Y \to Y$.
By taking suitable relative hyperplane sections $\{H_i\}_{i=1}^m$ on $X$,
we obtain a variety $Z \coloneqq X \cap H_1 \cap \cdots \cap H_m$
such that the induced morphism $\phi|_Z \colon Z \to Y$ is finite and surjective by equidimensionality.
Here $m \coloneqq  \dim X - \dim Y$.
The reflexive pull-back $\phi^{[*]}\mathcal{L}$ is then $\mathbb{Q}$-Cartier by assumption,
so $m_0 \phi^{*}(D)$ is a Cartier divisor for some $m_0 \in \mathbb{Z}_{+}$.

Since \(\phi\) is equidimensional, we can define the pull-backs $\phi^* D$ and $(\phi|_Z)^* D$ of Weil divisors, 
and $(\phi|_Z)^* D = (\phi^* D)|_Z$ holds. 
Since $m_0(\phi|_Z)^* D = m_0(\phi^* D)|_Z$ is Cartier and $\phi|_Z \colon Z \to Y$ is a finite morphism,
the Weil divisor $D$ is also $\mathbb{Q}$-Cartier by the norm argument in \cite[Lemma 5.16]{KM98}.

\smallskip
(2)
The conclusion is local on $Y$.
Taking an index-one cover $Y' \to Y$ of $Y$, which is quasi-\'etale,  and replacing $X$ (resp.\,$Z$) 
with the normalization of the main component of the fiber product $X \times_{Y} Y'$ (resp.\,$Z \times_{Y} Y'$),
we may assume that $K_Y$ is Cartier (cf.~\cite[Proposition 5.20]{KM98}). 
By construction, the variety $Z$ has klt singularities, and thus $Z$ has rational singularities.
Since the restriction $\phi|_{Z} \colon Z \to Y$ is a finite morphism, the base variety $Y$ also has rational singularities by \cite[Proposition 5.13]{KM98}.
Since $K_Y$ is Cartier, by \cite[Corollary 5.24]{KM98}, the variety $Y$ has canonical and hence klt singularities. 
\end{proof}

Lemma \ref{lem-vanish} and Theorem \ref{thm-flat} below are taken from the recent joint work of the first author, and will be used crucially  in the proof of our main results. 
Theorem \ref{thm-flat} is a flatness criterion for pseudo-effective sheaves, 
generalizing the corresponding result of \cite{HP19} from projective setting to K\"ahler varieties. 

\begin{lem}[{\cite[Proposition 4.7]{CDM}}]\label{lem-vanish}
Let $X$ be a normal compact analytic variety.
Let $\mathcal{L}$ be a rank-one torsion-free sheaf on $X$.
Assume that both $\mathcal{L}$ and its dual $\mathcal{L}^{\vee}$ are pseudo-effective as torsion-free sheaves.
Then, the first Chern class $c_{1}(\mathcal{L})$ vanishes.
\end{lem}

\begin{thm}[{\cite[Theorem 1.1]{CDM}}]\label{thm-flat}
Let \(X\) be a maximally quasi-\'etale compact klt K\"ahler variety.
Assume that \(\mathcal{E}\) is a pseudo-effective reflexive sheaf with vanishing first Chern class \(c_1(\mathcal{E})=0\).
Then \(\mathcal{E}\) is locally free and numerically flat on $X$.
\end{thm}

We close this section with Theorem \ref{thm-weak} below, 
which is a generalization of \cite[Theorem 4.6]{DGP24} to K\"ahler varieties;
its detailed proof will be given in the forthcoming paper of the first author \cite[Section 6]{MWWZ}.

\begin{thm}[{cf.~\cite[Theorem 4.6]{DGP24}}]\label{thm-weak}
Let $X$ be a strongly $\mathbb{Q}$-factorial klt compact K\"ahler variety, and
let $\mathcal{F}$ be a weakly regular foliation on $X$ with algebraic leaves.
Then $\mathcal{F}$ is induced by a surjective equidimensional morphism
$p \colon X \to Y$ onto a normal compact K\"ahler variety $Y$.
\end{thm}

\section{Proof of Main Results}

In this section, we prove the main results.
Let us begin with the setup below followed by a technical proposition that will be heavily used in the proofs.

\begin{setup}\label{set-mrc}
Let $X$ be a normal compact klt K\"ahler variety with pseudo-effective tangent sheaf.
Let $\phi \colon X \dashrightarrow Y$ be an almost holomorphic map onto a normal compact analytic variety $Y$,
and let $\pi \colon \widetilde X \to X$ be a resolution of the indeterminacy locus of $\phi \colon X \dashrightarrow Y$
and of the singularities of $X$, equipped with the following diagram:
\begin{equation*}
\xymatrix@C=40pt@R=30pt{
\widetilde X \ar[rr]^{\pi} \ar[rd]_{\varphi}  &  & X \ar@{-->}[ld]^{\phi}\\
& Y. &
}
\end{equation*}
Suppose that $\codim \pi(\varphi^{-1}(Y_{\sing})) \geq 2$ holds and that
the canonical divisor $K_{Y}$ is $\mathbb{Q}$-Cartier and pseudo-effective. 
\end{setup}

\begin{rem}\label{rem-wreg}
The technical assumptions in Setup \ref{set-mrc} are naturally satisfied in our applications.

\begin{enumerate}
\item[(1)] The assumption 
$
\codim \pi(\varphi^{-1}(Y_{\sing})) \geq 2
$ is satisfied if either \(Y\) is smooth, or $\phi \colon X \dashrightarrow Y$ is an everywhere-defined equidimensional fibration. 
\item[(2)] The assumption that the canonical divisor \(K_Y\) is pseudo-effective is satisfied if \(Y\) is non-uniruled.
Indeed, if we take a resolution of singularities \(\sigma\colon \widetilde{Y}\to Y\), then by the K\"ahler analogue of \cite{BDPP13}, which is recent breakthrough proved in \cite{Ou25}
$($cf.~\cite{CH20,CP}$)$, the canonical divisor  \(K_{\widetilde{Y}}\) is pseudo-effective; hence, \(K_Y=\sigma_*K_{\widetilde{Y}}\) is also pseudo-effective.
\item[(3)]  By (1) and (2), the assumptions in Setup \ref{set-mrc} are satisfied, if 
$\phi \colon X \dashrightarrow Y$ is either  an MRC fibration
onto a compact K\"ahler manifold $Y$ (see \cite{GHS03}), or an everywhere-defined equidimensional MRC fibration onto a \(\mathbb{Q}\)-factorial normal variety. 
\item[(4)] We will see in the proof of Theorem \ref{thm-str} that the assumptions in Setup \ref{set-mrc} are also satisfied if \(\phi\) is the MRC fibration to the minimal model of a compact K\"ahler manifold.
\end{enumerate}
\end{rem}

\begin{prop}\label{prop-flat}
Under Setup \ref{set-mrc}, the following statements hold:
\begin{itemize}
\item[$(1)$] The reflexive sheaf
$$
\mathcal{Q} \coloneqq (\pi_*{\varphi}^{*} \Omega_Y)^{\vee}
$$
is pseudo-effective with vanishing first Chern class.
In particular, if $X$ is maximally quasi-\'etale,
then $\mathcal{Q}$ is a numerically flat locally free sheaf on $X$.

\item[$(2)$] The pull-back $\varphi^{*}K_{Y}$ is numerically equivalent to an effective $\pi$-exceptional divisor.
In particular, the numerical dimension of $K_{Y}$ is equal to zero.
\end{itemize}
\end{prop}

\begin{proof}
The strategy of the proof is the same as in \cite{IMZ23}; the only difference is that we do not assume that $X$ is strongly $\mathbb{Q}$-factorial. Since the cotangent sheaf $\Omega_{Y}$ is locally free on $Y_{\reg}$, we have a natural injective morphism of sheaves
$
{\varphi}^{*} \Omega_Y \to \Omega_{\widetilde X}
$
on $\widetilde X \setminus \varphi^{-1}(Y_{\sing})$. 
Pushing forward by $\pi$, we obtain an injective morphism of sheaves
$
\pi_*{\varphi}^{*} \Omega_Y \to \pi_*\Omega_{\widetilde X}
$
on $X \setminus \pi(\varphi^{-1}(Y_{\sing}))$. 
By the assumption 
$
\codim \pi(\varphi^{-1}(Y_{\sing})) \geq 2, 
$
using $\Omega_X^{[1]} \cong \pi_{*}\Omega_{\widetilde X}$ and then  dualizing, 
we obtain an exact sequence
\begin{align}\label{eq-1}
0 \longrightarrow \mathcal{S}\coloneqq \Ker r
\longrightarrow T_{X}
\xrightarrow{\quad r \quad}
\mathcal{Q}\coloneqq (\pi_{*}{\varphi}^{*} \Omega_Y)^{\vee}
\qquad\text{on } X.
\end{align}
The morphism $T_{X} \to \mathcal{Q}$ is generically surjective by construction. Together with the pseudo-effectivity of $T_{X}$, this implies that  $\mathcal{Q}$ is also pseudo-effective (see, e.g., \cite[Proposition 2.4]{CDM}). 
It follows that $\det \mathcal{Q}$ is pseudo-effective as well. By reflexivity, we have
\[
\pi_{[*]}\mathcal{O}_{\widetilde{X}}(\varphi^{*}K_{Y})=\det \mathcal{Q}^{\vee}.
\]
By the assumption that the left-hand side is pseudo-effective,  
Lemma \ref{lem-vanish} then shows that $c_{1}(\mathcal{Q})$ vanishes. 
The remaining assertion of $(1)$ follows from Theorem \ref{thm-flat}. 

By construction, since $c_{1}(\mathcal{Q})=0$, 
we see that $M\coloneqq \pi^{[*]}\mathcal{Q}$ is numerically equivalent to a $\pi$-exceptional divisor. 
Since the difference between $\varphi^{*}K_{Y}$ and $M=\pi^{[*]}c_1(\mathcal{Q})$ is a $\pi$-exceptional divisor, 
the $\mathbb{Q}$-line bundle $\varphi^{*}K_{Y}$ is numerically equivalent to a $\pi$-exceptional divisor $F$. 
Since $\varphi^{*}K_{Y}$ is pseudo-effective, by  the divisorial Zariski decomposition, the divisor $F$ is effective (cf. the proof of \cite[Proposition 3.2]{MZ}). 
\end{proof}

\begin{proof}[Proof of Theorem \ref{thm-quot}]
The existence of maximally quasi-\'etale covers for klt K\"ahler varieties is known by  \cite[Theorem 2.5]{FO25} (cf.~\cite[Theorem 1.5]{GKP16} and \cite[Corollary  3.6]{Gue25}). 
Therefore, in view of Lemma \ref{lem-quot}, we may assume that $X$ is maximally quasi-\'etale.

Take an MRC fibration $\phi \colon X \dashrightarrow Y$ onto a compact K\"ahler manifold \(Y\).
By Remark \ref{rem-wreg} (3), the assumptions in Setup \ref{set-mrc} are satisfied, and hence  Proposition \ref{prop-flat} applies. 
For the reflexive sheaf $\mathcal{S}$ defined in  \eqref{eq-1},
we consider the torsion-free sheaf $\mathcal{Q}_{0}\coloneqq T_X/ \mathcal{S}$ with
\begin{align}\label{eq-2}
0 \longrightarrow \mathcal{S}\coloneqq \Ker r \longrightarrow
T_X \xrightarrow{\quad r \quad}\mathcal{Q}_{0} \longrightarrow 0
\qquad\text{on } X.
\end{align}
By \cite[Lemma 2.9]{IMZ23},
this sequence is an exact sequence of vector bundles
on the smooth locus $X_{\reg}$, noting that \(\mathcal{Q}\) is numerically flat by Proposition \ref{prop-flat}.
Thus, the reflexive hull $\mathcal{Q}_{0}^{\vee\vee}$ coincides with the numerically flat sheaf $\mathcal{Q}$.
The following claim directly follows from \cite[Proposition 3.1]{CDM}, 
but we give a sketch of the proof.

\begin{claim}\label{claim-ref}
$\mathcal{Q}_{0}=\mathcal{Q}$ on the whole variety $X$.
\end{claim}
\begin{proof}[Sketch of the proof of Claim \ref{claim-ref}]
To this end, it is enough to prove that $\mathcal{Q}_{0}$ is a reflexive sheaf on $X$.
Let $s\coloneqq \rank \mathcal{Q}_{0}$ and set
\[
\mathcal{F}\coloneqq \Lambda^{[s]}T_X[\otimes](\det \mathcal{Q}_{0})^{\vee}.
\]
Let $\tau \in H^0(X, \mathcal{F}^{\vee})$ be the natural section obtained from the inclusion
$
0 \to \mathcal{O}_{X} \to \mathcal{F}^{\vee}
$
induced by the exterior power of the morphism in \eqref{eq-2}.
Since $T_X$ is pseudo-effective and $c_{1}(\mathcal{Q})=0$, by \cite[Propositions 2.4 and 2.5]{CDM}, 
there exist singular Hermitian metrics $\{h_{m}\}$ on $\Sym^{m}\mathcal{F}$
such that
$$
\sqrt{-1}\Theta_{h_{m}}(\Sym^{m}\mathcal{F})
\geq - \omega_{X} \otimes \id \quad\text{on } X,
$$
where $\omega_{X}$ is a fixed K\"ahler form on $X$.
Let $X_0$ denote the maximal open subset on which the sheaves in \eqref{eq-2} are locally free and the sequence is an exact sequence of vector bundles.
The function
$
\log |\tau^{m}|_{h^{\vee}_{m}}
$
satisfies
$$
\sqrt{-1}\ddbar \Big( \frac{1}{m}\log |\tau^{m}|_{h^{\vee}_{m}} \Big) \geq - \frac{1}{m} \omega_{X}
$$
on $X_{0}$,
and hence extends to a quasi-psh function on $X$ with the same curvature inequality,
since $\codim (X\setminus X_{0}) \geq 2$ and $X$ is normal 
by \cite[Satz 4]{GR56}.
Set
\[
u_m\coloneqq\frac1m\log |\tau^m|_{h_m^\vee}.
\]
Since multiplying \(h_m\) by a positive constant does not change its
curvature, we normalize \(h_m\) so that
\(\sup_X u_m=0\). 
Then
\(dd^c u_m\geq-\frac1m\omega_X\). 
According to the relative compactness theorem for normalized quasi-psh functions, after
passing to a subsequence, \(u_m\) converges in \(L^1(X)\) to a
quasi-psh function \(u\). 
Passing to the limit in the preceding
curvature inequality yields \(dd^cu\geq0\). Thus \(u\) is psh, and
hence is constant because \(X\) is compact and connected.

Since $X$ has quotient singularities, by Lemma \ref{lem:q-f}, for every point  $x\in X$, there is an open neighborhood $x\in U \subset X$
and a finite quasi-\'etale cover $\nu\colon V \to U$
such that $V$ is smooth.
Taking the pull-back of the sequence (\ref{eq-2}) and its double dual, we obtain the induced exact sequence
\begin{align}\label{eq-orb}
0 \to \nu^{[*]}\mathcal{S} \to \nu^{[*]}T_{X}
\to \mathcal{Q}_{\tf}\coloneqq \nu^{[*]}T_{X}/\nu^{[*]}\mathcal{S}  \to 0
\quad\text{on } V.
\end{align}
Note that $\mathcal{Q}_{\tf}$ (as a subsheaf of \(\nu^{[*]}\mathcal{Q}_0\)) is torsion-free.
Further, the reflexive pull-back $\nu^{[*]}T_{X}$ is isomorphic to $T_{V}$;
in particular, it is locally free since $V$ is smooth.
We will show that $\mathcal{Q}_{\tf}$ is locally free and
that the exact sequence (\ref{eq-orb}) is an exact sequence of vector bundles.
We claim that $\tau_{\orb}$ is nowhere vanishing over $U$,
where $\tau_{\orb}$ is the natural section of
$$
\mathcal{G}\coloneqq \Lambda^{[s]} {T_V}[\otimes] (\det \mathcal{Q}_{\tf})^{\vee}
$$
induced by the exact sequence (\ref{eq-orb}).
We consider singular Hermitian metrics $\{h_{\orb, m}\}$ on $\Sym^{m}\mathcal{G}$
defined by pulling back $\{h_{m}\}$, namely $h_{\orb, m}\coloneqq \nu^{*}h_{m}$.
If $\tau_{\orb}$ has a zero at $p$,
the Lelong number of
$(1/m)\log |\tau^{m}_{\orb}|_{h_{\orb, m}^{\vee}}$ at $p$
is uniformly bounded from below.

Indeed, the section $\tau_{\orb}^{  m}$ can be locally written as 
$
\tau_{\orb}^{m}=\sum_I \tau_I e_I, 
$
where $\{e_i\}_{i=1}^r$ is a local frame of $\mathcal{G}$, 
$I$ is a multi-index of degree $m$, and $e_I:=\otimes_{i\in I} e_i$. 
The holomorphic function $\tau_I$ has the multiplicity $\geq m$ at $p$
from $\tau_{\orb}=0$. 
It can be seen that $|\langle e_I, e_J \rangle_{h^\vee_m}|$ is bounded 
since $\log | u |_{h^\vee_m}$ is almost psh for any local section $u$
(for example see \cite[Lemma 2.2.4]{PT18}). 
Therefore, we see that 
$$
| \tau_{\orb}^{  m} |_{h_{\orb, m}^{\vee}} \leq C \sum_I |\tau_I |. 
$$
This implies that the Lelong number is greater than or equal to one.

This function can be written as the pull-back of $(1/m) \log |\tau^{m}|_{h^{\vee}_{m}}$,
which is a contradiction.
By \cite[Lemma 1.20]{DPS94}, both $\nu^{[*]}\mathcal{S}$ and $\mathcal{Q}_{\tf}$ are locally free on \(U\). 

After taking a Galois closure, we may assume that $\nu \colon V \to U$ is a Galois cover with a Galois group $G$ so that $\nu^{[*]}\mathcal{S}$ and $\mathcal{Q}_{\tf}$ are still locally free.  
Then, by \cite[Lemma~A.3]{GKKP11}, their $G$-invariant pushforwards $(\nu_{*}(\bullet))^{G}$ are also reflexive over $U$.
Since $\nu^{[*]}\mathcal{S}$, $T_{V}$, and $\mathcal{Q}_{\tf}$
are the usual pull-backs of $\mathcal{S}$, $T_{X}$, and $\mathcal{Q}_{0}$ over $U_{\reg}$,
we obtain
$$
(\nu_{*}(\nu^{[*]}\mathcal{S}))^{G} = \mathcal{S}^{\vee\vee}=\mathcal{S}, \quad
(\nu_{*}T_{V})^{G} = T_{X}^{\vee\vee}=T_{X}, \quad
(\nu_{*}\mathcal{Q}_{\tf})^{G} = \mathcal{Q}_{0}^{\vee\vee}
$$
on $U$.
Together with the exactness of the functor $(\nu_{*}(\bullet))^{G}$ (see \cite[Lemma A.3]{GKKP11}),
we deduce the following short exact sequence
$$
0 \to \mathcal{S} \to T_{X} \to \mathcal{Q}_{0}^{\vee\vee} \to 0
\quad\text{on } U.
$$
This shows that $\mathcal{Q}_{0} \cong \mathcal{Q}_{0}^{\vee\vee}$ on $U$.
This completes the proof.
\end{proof}

We return to the proof of Theorem \ref{thm-quot}. 
By construction,
the sheaf $\mathcal{S} \subset T_{X}$ (defined by the dominant rational map \(\phi\)) is an algebraically integrable foliation.
Since $\mathcal{Q}_{0}$ is locally free by Claim \ref{claim-ref} and Proposition \ref{prop-flat},
the exact sequence \eqref{eq-2} splits locally.
Thus, the sheaf $\mathcal{S}$ is a weakly regular foliation.
Since $X$ is strongly $\mathbb{Q}$-factorial (see~\cite[Proposition 5.15]{KM98}), 
by Theorem \ref{thm-weak}, there exists an equidimensional holomorphic fibration $\phi' \colon X \to Y'$  
such that the foliation $\mathcal{S}$ is induced by this fibration. 
Moreover, \(Y'\) is \(\mathbb{Q}\)-factorial by Lemma \ref{lem-base}. 

Consider Setup \ref{set-mrc} for the equidimensional fibration $\phi' \colon X \to Y'$, 
which satisfies the assumptions of Setup \ref{set-mrc} by Remark \ref{rem-wreg}~(3). 
By Proposition \ref{prop-flat}, we consider the numerically flat locally free sheaf
$\mathcal{Q}'\coloneqq \phi'^{[*]}\Omega_{Y'}$ associated with $\phi' \colon X \to Y'$.  
Let $\rho \colon \pi_{1}(X) \to \GL(r,\mathbb{C})$ be a linear representation
associated with the flat vector bundle $\mathcal{Q}'$, 
where $r\coloneqq \rank \mathcal{Q}'$.
A general leaf of $\mathcal{S}$ is birational to a general fiber of the MRC fibration $\phi \colon X \dashrightarrow Y$,
and is thus rationally connected.
The fibration $\phi' \colon X \to Y'$ is therefore also the MRC fibration of $X$.
Hence, with \(Y\) and \(\phi\) replaced by \(Y'\) and \(\phi'\), we may assume that \(\phi\) is a holomorphic equidimensional MRC fibration. 
By \cite[Theorem 5.2]{Kol93} (or \cite[Theorem 4.1]{BC15}),
we have an isomorphism $\pi_{1}(X) \cong \pi_{1}(Y)$.
Using this isomorphism, we obtain a representation $\pi_{1}(Y) \to \GL(r,\mathbb{C})$
and the associated flat vector bundle $\mathcal{E}$.
By construction, we have $\mathcal{Q} \cong \phi^{*}\mathcal{E}$.
Since $\mathcal{Q}=\phi^{*}\Omega_{Y}$ over $X_{\reg}$,
it follows from reflexivity that $\Omega_{Y}^{[1]} \cong \mathcal{E}$,
and hence $T_Y$ is a flat locally free sheaf; in particular, 
by Lemma \ref{lem-base} and the known solution to Zariski--Lipman's conjecture \cite[Corollary 1.3]{GK14} for klt varieties, we obtain that \(Y\) is smooth. 
Thus, the base $Y$ is an \'etale quotient of a complex torus. 
Since \(\phi\) is an equidimensional holomorphic map from a klt variety to a compact K\"ahler manifold, \cite[Page 145 (Theorem) and Page 158 (Proposition)]{Fis76} and \cite[Corollary 5.25]{KM98} imply that \(\phi\) is flat. 
Furthermore, a very general fiber \(F\) has pseudo-effective tangent sheaf by the generic surjectivity of \(T_X|_F\to T_F\). 

We finally prove the asserted properties of the fibration using the functorial resolution. 
Let $\beta \colon \widehat{X} \to X$ be a functorial resolution of $X$.
Then, there is a generically surjective  morphism $T_{\widehat{X}} \to \beta^{*} \mathcal{Q}=\beta^*\phi^*T_Y$ induced by \(0\to\beta^*\phi^*\Omega_Y\to \Omega_{\widehat{X}}\).
Composed with the generically surjective map \(\beta^*\beta_*T_{\widehat{X}}\cong\beta^*T_X\to T_{\widehat{X}}\) induced by the evaluation   via the functorial resolution, we obtain the map \(s_1\colon \beta^*T_X\to \beta^*\phi^*T_Y\). 
On the other hand, pulling back \eqref{eq-2} via \(\beta\), we get a surjective morphism \(s_2\colon \beta^*T_X\to \beta^*\phi^*T_Y\).
Since \(s_1\) and \(s_2\) agree on a Zariski open set of \(\widehat{X}\), we have \(s_1=s_2\) and thus \(T_{\widehat{X}}\to \beta^*\phi^*T_Y\) is also surjective.
In particular, \(\widehat{X}\to Y\) is smooth. 
The desired properties of each fiber of \(\phi\) follow from the smoothness of \(\phi\circ\beta\), the flatness of \(\phi\) and the fact that \(X\) is Cohen-Macaulay (cf.~\cite[Lemma 5.2]{IMZ23}). 
\end{proof}

\begin{proof}[Proof of Theorem \ref{thm-str}]

Since $\widehat{q}(X) \leq \dim X$, 
after replacing $X$ with a finite quasi-\'etale cover, we may assume that $\widehat{q}(X)=q(X)$ and that $X$ is maximally quasi-\'etale.
Let $\phi \colon X \dashrightarrow Y$ be an MRC fibration onto a compact K\"ahler manifold.
By Remark \ref{rem-wreg} (3), all the assumptions in Setup \ref{set-mrc} are satisfied. 
By Proposition \ref{prop-flat}, the numerical dimension of $K_{Y}$ is equal to zero.
By Conjecture \ref{conj-min}, we take a minimal model $\alpha \colon Y \dashrightarrow Y_{\min}$ of $Y$ such that $Y_{\min}$ has terminal singularities and $K_{Y_{\min}}$ is numerically trivial.

After replacing $Y$ with a bimeromorphic model, we may assume that $\alpha \colon Y \dashrightarrow Y_{\min}$ is an everywhere-defined bimeromorphic morphism.
Consider the diagram: 
\begin{equation*} \xymatrix@C=40pt@R=30pt{ \widetilde X \ar[rr]^{\pi} \ar[rd]_{\varphi} \ar@/_50pt/[rrrd]^{\varphi_{\min}} & & X \ar@{-->}[ld]^{\phi} \ar@{-->}[rd]^{\phi_{\min}}&&\\ & Y \ar[rr]_{\alpha} && Y_{\min}.& } \end{equation*} 
Since $K_{Y_{\min}}$ is numerically trivial and $Y_{\min}$ has terminal singularities, 
there exists an effective $\alpha$-exceptional divisor $N_{Y}$ on $Y$ such that
$$
K_{Y} \sim_{\mathbb{Q}} \alpha^{*}K_{Y_{\min}}+ N_{Y} \equiv_{\num} N_{Y}
$$
and 
the support of $N_{Y}$ coincides with the $\alpha$-exceptional locus. 
In  particular, the support of $N_{Y}$ contains $\alpha^{-1}(Y_{\min, \sing})$.
Applying Remark \ref{rem-wreg} (3) and  Proposition \ref{prop-flat} to $\phi \colon X \dashrightarrow Y$, we see that $\varphi^{*}N_{Y}$ is numerically equivalent to a $\pi$-exceptional divisor.
Thus the support of $\varphi^{*}N_{Y}$ is contained in the $\pi$-exceptional locus, which shows that
\[
\codim \pi (\varphi^{-1}\alpha^{-1}(Y_{\min, \sing})) \geq 2.
\]
Consequently, the MRC fibration $\phi_{\min} \colon X \dashrightarrow Y_{\min}$ satisfies the assumptions of Setup \ref{set-mrc} (cf.~Remark \ref{rem-wreg} (4)).

For the minimal model \(Y_{\textup{min}}\), the singular Beauville--Bogomolov--Yau decomposition (see \cite{BGL22,HP19}; cf.~\cite{Dru18,GKP16b,GGK19}) shows that there exists a finite quasi-\'etale cover
$$
\nu \colon T \times Z \to Y_{\min},
$$
where $T$ is a complex  torus and $Z$ is a product of singular hyperK\"ahler varieties and strict Calabi--Yau varieties.
Note that $Z$ has vanishing augmented irregularity $\widehat{q}(Z) = 0$.
Hereafter, for simplicity of notation, we write $\phi$ and $\varphi$ for $\phi_{\min}$ and $\varphi_{\min}$, respectively.
Moreover, let $\Gamma$ be the normalization of the main component of the closure of the graph of $\phi$.
\begin{equation*} \xymatrix@C=40pt@R=30pt{ \Gamma \ar[rr]^{\pi} \ar[rd]_{\varphi} & & X \ar@{-->}[ld]^{\phi}\\ & Y_{\min}. & } \end{equation*}
Let $I_{\phi}$ be the indeterminacy locus of the MRC fibration $\phi \colon X \dashrightarrow Y_{\min}$.
Consider the induced morphism
\[
\phi_{0} \colon X_{0}\coloneqq X \setminus I_{\phi} \to Y_{0}\coloneqq Y_{\min} \setminus \varphi(\pi^{-1}(I_{\phi}))
\]
and the fiber product, which fits into the following Cartesian diagram: 
\begin{equation*} \xymatrix@C=45pt@R=32pt{ X'_{0} \ar[r] & X_{0} \times_{Y_{0}} \nu^{-1}(Y_{0}) \ar[r] \ar[d] & X_{0} \subset X \ar[d]^{\phi_{0}} \\ & \nu^{-1}(Y_{0}) \ar[r]^{\nu|_{\nu^{-1}(Y_{0})}} & Y_{0} \subset Y_{\min}, } \end{equation*}
where $X'_{0}$ is the normalization of $X_{0} \times_{Y_{0}} \nu^{-1}(Y_{0})$.
Note that \(X\) is klt and \(Y_{\textup{min}}\) is not uniruled; hence, it follows from \cite{HM07} that \(Y_0\) is a non-empty open subset. 
Since
$$
\pi \bigl(\varphi^{-1}(Y_{\min, \sing}\cap Y_{0})\bigr)
= \phi_{0}^{-1} (Y_{\min, \sing} \cap Y_{0})
$$
has codimension at least $2$, the morphism $X'_{0} \to X_{0}$ is a finite quasi-\'etale cover.
Moreover, since $\codim I_{\phi} \geq 2$, we have
$$
\pi_{1}(X_{0, \reg}) \cong \pi_{1}(X_{\reg}).
$$
Thus, the subgroup of $\pi_{1}(X_{0, \reg})$ corresponding to $X'_{0} \to X_{0}$ determines a finite-index subgroup of $\pi_{1}(X_{\reg})$.
This shows that the finite quasi-\'etale cover $X'_{0} \to X_{0}$ extends to a finite quasi-\'etale cover $X' \to X$.
Thus, after replacing $X$ with $X'$, we may assume that $Y_{\min}$ is isomorphic to $T \times Z$.

We claim that the composite map 
$$
X \dashrightarrow Y_{\min}= T\times Z \to Z
$$
is also almost holomorphic.
Take the normalization of the graph \(\Gamma'\) with the induced projections \(p'\colon \Gamma'\to X\) and \(q'\colon \Gamma'\to Z\).
If the composite map is not almost holomorphic, some exceptional component \(E\) of \(p'\) would dominate \(Z\).
Since \(E\) is rationally connected by \cite{Fuj26} or \cite{HM07}, it follows that \(Z\) is rationally connected as well.
On the other hand, since \(Z\)  has only canonical singularities, by \cite{Ou25}, they are not covered by rational curves, which is a contradiction.

Now, we regard the induced almost holomorphic map \(X\dashrightarrow Z\) 
as $\phi \colon X \dashrightarrow Y$ in Setup \ref{set-mrc}.
We consider the numerically flat locally free sheaf $\mathcal{Q}$ constructed in Proposition \ref{prop-flat} from the cotangent sheaf $\Omega_{Z}$; in particular, \(\textup{rank}\mathcal{Q}=\dim Z>0\). 
Note that a dominant rational map from a compact klt K\"ahler variety to a complex torus is indeed well-defined. 
Hence, the first half of Theorem \ref{thm-str} is proved if we can show that \(Z\) is a single point; then the second half follows from the generic surjectivity of \(T_X|_F\to T_F\) with \(F\) a very general fiber.

Assume to the contrary that \(Z\) is of positive dimension.
Then \(\textup{rank}\mathcal{Q}>\dim T\). 
Let $\rho \colon \pi_{1}(X) \to \GL(r,\mathbb{C})$ be the representation corresponding to the flat locally free sheaf $\mathcal{Q}$.
Since a general fiber of $\varphi\colon \Gamma \to Y_{\min}$ (birational to that of \(\phi\)) is rationally connected, 
we obtain an isomorphism
\[
\pi_{1}(X) \xleftarrow[\cong]{\quad \pi_{*}\quad} \pi_{1}(\Gamma) \xleftarrow[\cong]{\quad \varphi_{*}\quad } \pi_{1}(Y_{\min}) \cong \pi_{1}(T) \times \pi_{1}(Z)
\]
by \cite{Tak03} and \cite[Theorem 5.2]{Kol93} (or \cite[Theorem 4.1]{BC15}).
Therefore, \[
q(X)=\dim T.\]
The flat locally free sheaf $\mathcal{E}$ defined by
$$
\pi_{1}(T) \times \pi_{1}(Z) \cong \pi_{1}(X) \xrightarrow{\quad \rho \quad} \GL(r,\mathbb{C})
$$
satisfies $\varphi^{*}\mathcal{E}=\pi^{*}\mathcal{Q}$ by construction.
Consider the induced linear representation
$$
\tau \colon \pi_{1}(Z) \longrightarrow { \id_{\pi_{1}(T)} }\times \pi_{1}(Z) \subset \pi_{1}(X) \xrightarrow{\quad \rho \quad} \GL(r,\mathbb{C}).
$$
Since $T_X$ is pseudo-effective, it follows from \cite[Proposition 4.1]{MZ} that 
the image $\Image {\tau}$ of this representation is virtually abelian.  
As $q(Z)=0$, it follows that $\Image {\tau}$ is finite. 
Consider the finite \'etale cover $Z' \to Z$ corresponding to $\Ker {\tau}$, and take the fiber product with respect to $Y_{\min} \to Z$ and $X \dashrightarrow Z$.
Although the almost holomorphic map $X \dashrightarrow Z$ need not be everywhere defined, 
the corresponding fiber product can be taken in the same way as above.
Thus, we may assume that $\rho \colon \pi_{1}(Y_{\min}) \to \GL(r,\mathbb{C})$ is induced by a representation of $\pi_{1}(T)$.

Let $p \colon T \times Z \to T$ and $q \colon T \times Z \to Z$ be the projections.
Then there exists a flat locally free sheaf $\mathcal{F}$ on $T$ such that $p^{*}\mathcal{F}=\mathcal{E}$.
However, by construction, for a general fiber $T\times \{z\}$ of $q \colon T \times Z \to Z$, the restriction $\mathcal{E}|_{T\times \{z\}}$ is the trivial vector bundle.
Via the isomorphism $p \colon T\times \{z\} \to T$,  we deduce that $\mathcal{F}$, and hence $p^{*}\mathcal{F}=\mathcal{E}$, is the trivial vector bundle.
This implies that $\mathcal{Q}^{\vee} \subset \Omega_{X}^{[1]}$ is a trivial subvector bundle. 
Together with the pull-back of \(\Omega_T\), it gives a generically injective map
\[
\mathcal{O}_X^{\oplus\dim T}\oplus\mathcal{Q}^{\vee}\to\Omega_X^{[1]}
\]
because on the locus where \(X\dashrightarrow T\times Z\) is defined the two summands come from the two summands of \(\Omega_{T\times Z}=p^*\Omega_T\oplus q^*\Omega_Z\).  
This yields a contradiction by noting that 
\begin{align*}
\dim T&=q(X)=h^1(X,\mathcal{O}_X)=h^1(\Gamma,\mathcal{O}_{\Gamma})=h^0(\Gamma,\Omega_{\Gamma}^1)=h^0(X,\pi_*\Omega_{\Gamma}^1)\\
&=h^0(X,\Omega_X^{[1]})\geq \textup{rank}\mathcal{Q}+\dim T=\dim T+\dim Z>\dim T.
\end{align*}
This completes the proof.
\end{proof}

\begin{rem}\label{rem-minimal}
The required existence of such a minimal model is known in the projective case by the abundance theorem for numerically trivial canonical divisors (see \cite{Gon13}), and in the low-dimensional K\"ahler cases (see \cite{CHP16, DHP24,DH25}).
The existence of minimal models is used only at this point. Hence, Corollary \ref{cor-str} follows directly from Theorem \ref{thm-str}.
\end{rem}

We end this section with the following corollary, whose proof follows immediately from the proof of Theorem \ref{thm-str}.
\begin{cor}\label{cor-int}
Let \(f\colon X\to X\) be an int-amplified endomorphism of a klt compact K\"ahler variety.
Assume that Question \ref{q_mrc} has an affirmative answer for \(f\).
Then up to a quasi-\'etale cover, \(X\) admits a flat MRC fibration onto a complex torus such that every fiber is reducible and irreducible, and a very general fiber has pseudo-effective tangent sheaf. 
\end{cor}
\begin{proof}
Let \(\phi\colon X\dashrightarrow Y\) be the dominant meromorphic map as in Question \ref{q_mrc} such that after iterates, \(f\) descends to an int-amplified endomorphism \(g\) on \(Y\).
Since \(Y\) is non-uniruled,  it follows from \cite[Proposition 3.2]{MZ} that \(Y\) is a quasi-\'etale quotient of a complex torus.
We aim to show that \(\pi\) satisfies the assumptions in Setup \ref{set-mrc}.
It suffices to show that \(\textup{codim}\,\pi(\varphi^{-1}(Y_{\sing}))\geq 2\) in the commutative diagram below
\begin{equation*}
\xymatrix@C=40pt@R=30pt{
\widetilde X \ar[rr]^{\pi} \ar[rd]_{\varphi}  &  & X \ar@{-->}[ld]^{\phi}\\
& Y. &
}
\end{equation*}
Here, it is sufficient to consider the case when \(\widetilde{X}\) is the normalization of the graph of \(\pi\).
By the universal property of the graph, \(f\) and \(g\) lift to a surjective endomorphism \(h\) on \(\widetilde{X}\), which is also int-amplified.
As proved in \cite[Lemma 4.3]{Meng20}, the closed subset of \(Y\) over which \(\varphi\) is not equidimensional, is \(g^{-1}\)-invariant; hence, it is empty due to \cite[Lemma 5.4]{Zho21}.
Since \(\pi\) is a birational morphism, it follows that  \(\textup{codim}\,\pi(\varphi^{-1}(Y_{\sing}))\geq 2\) which concludes the desired claim.

By Theorem \ref{thm-str} and \cite[Theorem 1.4]{MZ}, up to an \(f\)-equivariant quasi-\'etale cover, \(X\) admits a holomorphic MRC fibration \(\chi\colon X\to T\) onto a complex torus \(T\) such that \(f\) descends to an int-amplified endomorphism of \(T\).
By \cite[Proof of Lemma 4.3]{Meng20} again, the closed subset \(S\) of \(Y\) over which \(\varphi\) has reducible, or non-reduced, or non-equidimensional fibers is \(g^{-1}\)-invariant.
Therefore, \cite[Lemma 5.4]{Zho21} implies that 
\(S\) is empty. 
Consequently, \(\chi\) is equidimensional and hence flat by the miracle flatness.
\end{proof}
\begin{rem}
Motivated by \cite[Theorem 5.1 (3)]{MZ}, we may also ask whether a general fiber of the structure in Corollary \ref{cor-int} is maximally quasi-\'etale, in which case, a general periodic fiber is of Fano type.
\end{rem}

\bibliographystyle{amsalpha}

\begin{thebibliography}{99}
\bibitem[BGL22]{BGL22}
Benjamin Bakker, Henri Guenancia, and Christian Lehn.
\newblock Algebraic approximation and the decomposition theorem for K\"ahler Calabi--Yau varieties.
\newblock {\em Invent. Math.} \textbf{228} (2022), no.~3, 1255--1308.

\bibitem[BKK+15]{BKK}
Thomas Bauer, S\'andor J. Kov\'acs, Alex K\"uronya, Ernesto Carlo Mistretta, Tomasz Szemberg, and Stefano Urbinati.
\newblock On positivity and base loci of vector bundles.
\newblock {\em Eur. J. Math.} \textbf{1} (2015), no.~2, 229--249.





\bibitem[BDPP13]{BDPP13}
S\'ebastien Boucksom, Jean-Pierre Demailly, Mihai P\u{a}un, and Thomas Peternell.
\newblock The pseudo-effective cone of a compact K\"ahler manifold and varieties of negative Kodaira dimension.
\newblock {\em J. Algebraic Geom.} \textbf{22} (2013), no.~2, 201--248.

\bibitem[BC15]{BC15}
Yohan Brunebarbe and Fr\'ed\'eric Campana.
\newblock Fundamental group and pluridifferentials on compact K\"ahler manifolds.
\newblock {\em Mosc. Math. J.} \textbf{16} (2016), no.~4, 651--658.

\bibitem[CCM21]{CCM21}
Fr\'ed\'eric Campana, Junyan Cao, and Shin-ichi Matsumura.
\newblock Projective klt pairs with nef anti-canonical divisor.
\newblock {\em Algebr. Geom.} \textbf{8} (2021), no.~4, 430--464.

\bibitem[CHP16]{CHP16}
Fr\'ed\'eric Campana, Andreas H\"oring, and Thomas Peternell.
\newblock Abundance for K\"ahler threefolds.
\newblock {\em Ann. Sci. \'Ec. Norm. Sup\'er. (4)} \textbf{49} (2016), no.~4, 971--1025.

\bibitem[CP91]{CP91}
Fr\'ed\'eric Campana and Thomas Peternell.
\newblock Projective manifolds whose tangent bundles are numerically effective.
\newblock {\em Math. Ann.} \textbf{289} (1991), no.~1, 169--187.

\bibitem[Cao19]{Cao19}
Junyan Cao.
\newblock Albanese maps of projective manifolds with nef anticanonical bundles.
\newblock {\em Ann. Sci. \'Ec. Norm. Sup\'er. (4)} \textbf{52} (2019), no.~5, 1137--1154.

\bibitem[CDM]{CDM}
Junyan Cao, Ya Deng, and Shin-ichi Matsumura.
\newblock A flatness criterion for pseudo-effective sheaves on compact K\"ahler spaces.
\newblock Preprint, arXiv:2609.05154, 2026.


\bibitem[CH19]{CH19}
Junyan Cao and Andreas H\"oring.
\newblock A decomposition theorem for projective manifolds with nef anticanonical bundle.
\newblock {\em J. Algebraic Geom.} \textbf{28} (2019), no.~3, 567--597.

\bibitem[CH20]{CH20}
Junyan Cao and Andreas H\"oring.
\newblock Rational curves on compact K\"ahler manifolds.
\newblock {\em J. Differential Geom.} \textbf{114} (2020), no.~1, 1--39.

\bibitem[CP]{CP}
Junyan Cao and Mihai P\u{a}un.
\newblock Remarks on relative canonical bundles and algebraicity criteria for foliations in the K\"ahler context.
\newblock Preprint, arXiv:2502.02183, 2025.


\bibitem[DH25]{DH25}
Omprokash Das and Christopher D. Hacon.
\newblock The log minimal model program for K\"ahler $3$-folds.
\newblock {\em J. Differential Geom.} \textbf{130} (2025), no.~1, 151--207.

\bibitem[DHP24]{DHP24}
Omprokash Das, Christopher D. Hacon, and Mihai P\u{a}un.
\newblock On the $4$-dimensional minimal model program for K\"ahler varieties.
\newblock {\em Adv. Math.} \textbf{443} (2024), Paper No.~109615, 68 pp.

\bibitem[Dem12]{Dem12}
Jean-Pierre Demailly.
\newblock {\em Analytic Methods in Algebraic Geometry}.
\newblock Surveys of Modern Mathematics, vol.~1, International Press, Somerville, MA; Higher Education Press, Beijing, 2012.

\bibitem[DPS94]{DPS94}
Jean-Pierre Demailly, Thomas Peternell, and Michael Schneider.
\newblock Compact complex manifolds with numerically effective tangent bundles.
\newblock {\em J. Algebraic Geom.} \textbf{3} (1994), no.~2, 295--345.


\bibitem[Dru18]{Dru18}
St\'ephane Druel.
\newblock A decomposition theorem for singular spaces with trivial canonical class of dimension at most five.
\newblock {\em Invent. Math.} \textbf{211} (2018), no.~1, 245--296.


\bibitem[DGP24]{DGP24}
St\'ephane Druel, Henri Guenancia, and Mihai P\u{a}un.
\newblock A decomposition theorem for $\mathbb{Q}$-Fano K\"ahler--Einstein varieties.
\newblock {\em C. R. Math. Acad. Sci. Paris} \textbf{362} (2024), 93--118.

\bibitem[Fis76]{Fis76}
Gerd Fischer.
\newblock {\em Complex Analytic Geometry}.
\newblock Lecture Notes in Mathematics, vol.~538, Springer-Verlag, Berlin--New York, 1976.




\bibitem[FO]{FO25}
Xin Fu and Wenhao Ou.
\newblock Orbifold Bogomolov--Gieseker inequalities on compact K\"ahler varieties.
\newblock Preprint, arXiv:2511.03530, 2025.


\bibitem[Fuj]{Fuj26}
Osamu Fujino.
\newblock Notes on rational chain connectedness.
\newblock Preprint, arXiv:2602.19415, 2026.

\bibitem[Gon11]{Gon11}
Yoshinori Gongyo.
\newblock On the minimal model theory for dlt pairs of numerical log Kodaira dimension zero.
\newblock {\em Math. Res. Lett.} \textbf{18} (2011), no.~5, 991--1000.

\bibitem[Gon13]{Gon13}
Yoshinori Gongyo.
\newblock Abundance theorem for numerically trivial log canonical divisors of semi-log canonical pairs.
\newblock {\em J. Algebraic Geom.} \textbf{22} (2013), no.~3, 549--564.

\bibitem[GHS03]{GHS03}
Tom Graber, Joe Harris, and Jason Starr.
\newblock Families of rationally connected varieties.
\newblock {\em J. Amer. Math. Soc.} \textbf{16} (2003), no.~1, 57--67.

\bibitem[GK14]{GK14}
Patrick Graf and S\'andor J. Kov\'acs.
\newblock An optimal extension theorem for 1-forms and the {L}ipman-{Z}ariski conjecture.
\newblock {\em Doc. Math.} \textbf{19} (2014), 815--830.




\bibitem[GGK19]{GGK19}
Daniel Greb, Henri Guenancia, and Stefan Kebekus.
\newblock Klt varieties with trivial canonical class: holonomy, differential forms, and fundamental groups.
\newblock {\em Geom. Topol.} \textbf{23} (2019), no.~4, 2051--2124.

\bibitem[GKKP11]{GKKP11}
Daniel Greb, Stefan Kebekus, S\'andor J. Kov\'acs, and Thomas Peternell.
\newblock Differential forms on log canonical spaces.
\newblock {\em Publ. Math. Inst. Hautes \'Etudes Sci.} \textbf{114} (2011), 87--169.

\bibitem[GKP16a]{GKP16}
Daniel Greb, Stefan Kebekus, and Thomas Peternell.
\newblock \'Etale fundamental groups of Kawamata log terminal spaces, flat sheaves, and quotients of abelian varieties.
\newblock {\em Duke Math. J.} \textbf{165} (2016), no.~10, 1965--2004.

\bibitem[GKP16b]{GKP16b}
Daniel Greb, Stefan Kebekus, and Thomas Peternell.
\newblock Singular spaces with trivial canonical class.
\newblock In {\em Minimal Models and Extremal Rays}, Adv. Stud. Pure Math., vol.~70, Math. Soc. Japan, Tokyo, 2016, pp.~67--113.

\bibitem[Gue]{Gue25}
Henri Guenancia.
\newblock Beauville--Bogomolov decomposition for klt varieties.
\newblock Preprint, arXiv:2509.10053, 2025.

\bibitem[GR56]{GR56}
H.~Grauert and R.~Remmert. 
\newblock Plurisubharmonische Funktionen in komplexen R{\"a}umen. 
\newblock {\em Math. Z.} \textbf{65} (1956), 175--194.


\bibitem[HM07]{HM07}
Christopher D. Hacon and James McKernan.
\newblock On Shokurov's rational connectedness conjecture.
\newblock {\em Duke Math. J.} \textbf{138} (2007), no.~1, 119--136.

\bibitem[HW20]{HW20}
Gordon Heier and Bun Wong.
\newblock On projective K\"ahler manifolds of partially positive curvature and rational connectedness.
\newblock {\em Doc. Math.} \textbf{25} (2020), 219--238.

\bibitem[HP19]{HP19}
Andreas H\"oring and Thomas Peternell.
\newblock Algebraic integrability of foliations with numerically trivial canonical bundle.
\newblock {\em Invent. Math.} \textbf{216} (2019), no.~2, 395--419.

\bibitem[HIM22]{HIM22}
Genki Hosono, Masataka Iwai, and Shin-ichi Matsumura.
\newblock On projective manifolds with pseudo-effective tangent bundle.
\newblock {\em J. Inst. Math. Jussieu} \textbf{21} (2022), no.~5, 1801--1830.

\bibitem[HSW81]{HSW81}
Alan Howard, Brian Smyth, and Hung-Hsi Wu.
\newblock On compact K\"ahler manifolds of nonnegative bisectional curvature I and II.
\newblock {\em Acta Math.} \textbf{147} (1981), no.~1--2, 51--70.

\bibitem[IMZ]{IMZ23}
Masataka Iwai, Shin-ichi Matsumura, and Guolei Zhong.
\newblock Positivity of tangent sheaves of projective varieties: the structure of MRC fibrations.
\newblock To appear in {\em Algebraic Geometry}, arXiv:2309.09489.

\bibitem[Kol93]{Kol93}
J\'anos Koll\'ar.
\newblock Shafarevich maps and plurigenera of algebraic varieties.
\newblock {\em Invent. Math.} \textbf{113} (1993), no.~1, 177--215.

\bibitem[KM98]{KM98}
J\'anos Koll\'ar and Shigefumi Mori.
\newblock {\em Birational Geometry of Algebraic Varieties}.
\newblock Cambridge Tracts in Mathematics, vol.~134, Cambridge University Press, Cambridge, 1998.

\bibitem[Mat22]{Mat22a}
Shin-ichi Matsumura.
\newblock On projective manifolds with semipositive holomorphic sectional curvature.
\newblock {\em Amer. J. Math.} \textbf{144} (2022), no.~3, 747--777.

\bibitem[Mat23]{Mat23}
Shin-ichi Matsumura.
\newblock On the minimal model program for projective varieties with pseudo-effective tangent sheaf.
\newblock {\em \'Epijournal G\'eom. Alg\'ebrique} \textbf{7} (2023), Art. 22, 13 pp.

\bibitem[Mat26]{Mat26}
Shin-ichi Matsumura.
\newblock Fundamental group.
\newblock In preparation.

\bibitem[MQ26]{MQ26}
Shin-ichi Matsumura and Chenghao Qing.
\newblock On compact K\"ahler manifolds with pseudo-effective tangent bundle.
\newblock {\em International Mathematics Research Notices} \textbf{2026} (2026), no.~10, rnag096.

\bibitem[MW25]{MW25}
Shin-ichi Matsumura and Juanyong Wang.
\newblock Structure theorem for projective klt pairs with nef anti-canonical divisor.
\newblock {\em J. Eur. Math. Soc.}, published online first, 2025.
\newblock doi:10.4171/JEMS/1702.

\bibitem[MWWZ]{MWWZ}
Shin-ichi Matsumura, Juanyong Wang, Xiaojun Wu, and Qimin Zhang.
\newblock Beauville--Bogomolov--Yau decomposition for K\"ahler generalized pairs. 
\newblock in preparation.

\bibitem[MZ]{MZ}
Shin-ichi Matsumura and Guolei Zhong.
\newblock Positivity of compact K\"ahler varieties admitting an int-amplified endomorphism.
\newblock In preparation.

\bibitem[MY21]{MY21}
Yohsuke Matsuzawa and Shou Yoshikawa.
\newblock Int-amplified endomorphisms on normal projective surfaces.
\newblock {\em Taiwanese J. Math.} \textbf{25} (2021), no.~4, 681--697.

\bibitem[Men20]{Meng20}
Sheng Meng.
\newblock Building blocks of amplified endomorphisms of normal projective varieties.
\newblock {\em Math. Z.} \textbf{294} (2020), 1727--1747.

\bibitem[MZ20]{MZ20}
Sheng Meng and De-Qi Zhang.
\newblock Semi-group structure of all endomorphisms of a projective variety admitting a polarized endomorphism.
\newblock {\em Math. Res. Lett.} \textbf{27} (2020), no.~2, 523--549.

\bibitem[Mok88]{Mok88}
Ngaiming Mok.
\newblock The uniformization theorem for compact K\"ahler manifolds of nonnegative holomorphic bisectional curvature.
\newblock {\em J. Differential Geom.} \textbf{27} (1988), no.~2, 179--214.

\bibitem[Mor79]{Mor79}
Shigefumi Mori.
\newblock Projective manifolds with ample tangent bundles.
\newblock {\em Ann. of Math. (2)} \textbf{110} (1979), no.~3, 593--606.

\bibitem[M{\"u}l25]{Mul25}
Niklas M{\"u}ller.
\newblock Locally constant fibrations and positivity of curvature.
\newblock {\em Bull. Lond. Math. Soc.} \textbf{57} (2025), no.~4, 1005--1025.

\bibitem[Nak10]{Nak10}
Noboru Nakayama.
\newblock Intersection sheaves over normal schemes.
\newblock {\em J. Math. Soc. Japan} \textbf{62} (2010), no.~2, 487--595.

\bibitem[Ou]{Ou25}
Wenhao Ou.
\newblock A characterization of uniruled compact K\"ahler manifolds.
\newblock Preprint, arXiv:2501.18088, 2025.


 \bibitem[PT18]{PT18} 
Mihai P\u{a}un and Shigeharu Takayama. 
\newblock Positivity of twisted relative pluricanonical divisors and their direct images. 
\newblock {\em J. Algebraic Geom.} {\bf{27}} (2018), 211--272.


\bibitem[SY80]{SY80}
Yum-Tong Siu and Shing-Tung Yau.
\newblock Compact K\"ahler manifolds of positive bisectional curvature.
\newblock {\em Invent. Math.} \textbf{59} (1980), no.~2, 189--204.

\bibitem[Tak03]{Tak03}
Shigeharu Takayama.
\newblock Local simple connectedness of resolutions of log-terminal singularities.
\newblock {\em Internat. J. Math.} \textbf{14} (2003), no.~8, 825--836.

\bibitem[Wan21]{Wan21}
Juanyong Wang.
\newblock On the Iitaka conjecture $C_{n,m}$ for K\"ahler fiber spaces.
\newblock {\em Ann. Fac. Sci. Toulouse Math. (6)} \textbf{30} (2021), no.~4, 813--897.

\bibitem[Yan18]{Yan18}
Xiaokui Yang.
\newblock RC-positivity, rational connectedness and Yau's conjecture.
\newblock {\em Camb. J. Math.} \textbf{6} (2018), no.~2, 183--212.

\bibitem[Yos21]{Yos21}
Shou Yoshikawa.
\newblock Structure of Fano fibrations of varieties admitting an int-amplified endomorphism.
\newblock {\em Adv. Math.} \textbf{391} (2021), Paper No.~107964, 32 pp.

\bibitem[Zho21]{Zho21}
Guolei Zhong.
\newblock Int-amplified endomorphisms of compact K\"ahler spaces.
\newblock {\em Asian J. Math.} \textbf{25} (2021), no.~3, 369--392.



\end{thebibliography}

\end{document}